\documentclass[a4paper,12pt]{article}
\usepackage{amsmath,amssymb,amsthm}
\usepackage[top=2.5cm,bottom=2cm,left=2.5cm,right=2.5cm]{geometry}
\usepackage[margin=1cm,%
font=small,%
format=hang,%
labelsep=period,%
labelfont=bf]{caption}

\usepackage{graphicx}
\usepackage{epstopdf}

\renewcommand{\thefootnote}{\fnsymbol{footnote}}

\usepackage{enumerate}
\usepackage{hyperref}
\usepackage{comment}
\usepackage[normalem]{ulem}
\excludecomment{longver}

\newtheorem{lma}{Lemma}
\newtheorem{thm}{Theorem}
\newtheorem{cor}{Corollary}

\newcommand{\vek}[1]{\mathbf{#1}}

\DeclareMathOperator{\supp}{supp}

\begin{document}
	
	\thispagestyle{empty}
	
	\begin{center}
		
		{\Large \bf On binary codes with distances $d$ and $d+2$}
		
		\vspace{3mm}
		
		Ivan Landjev$^{a,b}$ and Konstantin Vorob'ev$^{b}$
		
		\vspace{3mm}

		\let\thefootnote\relax\footnotetext{E-mail addresses: \texttt{i.landjev@nbu.bg (Ivan Landjev), konstantin.vorobev@gmail.com (Konstantin Vorob'ev)}}
		{\it $^a$New Bulgarian University, 21 Montevideo str., 1618 Sofia, Bulgaria \\
			$^b$Institute of Mathematics and Informatics,
			Bulgarian Academy of Sciences,\\ 8 Acad G. Bonchev str., 1113 Sofia, Bulgaria}
		
	\end{center}
	
	\vspace{3mm}
	
	\begin{abstract}
		We consider the problem of finding 
		$A_2(n,\{d_1,d_2\})$ defined as the maximal size of a 
		binary (non-linear) code of length $n$ with two distances
		$d_1$ and $d_2$. Binary codes with distances $d$ and $d+2$ 
		of size $\sim\frac{n^2}{\frac{d}{2}(\frac{d}{2}+1)}$ can be obtained from $2$-packings 
		of an $n$-element set by blocks of cardinality $\frac{d}{2}+1$.
		This value is far from the upper bound
		$A_2(n,\{d_1,d_2\})\le1+{n\choose2}$ proved recently by Barg et al. 
		
		In this paper we prove that for every fixed $d$ ($d$ even)
		there exists an integer $N(d)$ such that for every
		$n\ge N(d)$ it holds
		$A_2(n,\{d,d+2\})=D(n,\frac{d}{2}+1,2)$,
		or, in other words, optimal codes are isomorphic to constant weight codes.
		We prove also estimates on $N(d)$ for $d=4$ and $d=6$.
	\end{abstract}
	


\section{Preliminaries}
\label{sec:preliminaries}

In this paper we consider binary non-linear codes with two distances.
The alphabet throughout the paper will be $\Omega=\{0,1\}$. A binary code
of length $n$ and size $M$, or an $(n,M)$-code,
is  a subset $C$ of $\Omega^n$.
A code $C$ is said to be a two-distance code if
\[\{d(\vek{u},\vek{v})\mid \vek{u},\vek{v}\in C,\vek{u}\ne\vek{v}\}=\{d_1,d_2\},\]
where $1\le d_1<d_2\le n$ are integers. A two-distance
binary code of length $n$, size $M$, and distances $d_1$ and $d_2$ is 
referred to as an $(n,M,\{d_1,d_2\})$-code.
If the size of the code is not specified we speak of an $(n,\{d_1,d_2\})$-code.
The maximum size $M$ of a binary code of length $n$
with distances $d_1$ and $d_2$ is denoted by $A_2(n,\{d_1,d_2\})$.

A natural problem is to determine the exact value of  $A_2(n,\{d_1,d_2\})$
for given positive integers $n, d_1$ and $d_2$. Barg et al. proved in \cite{Barg-et-al} 
that $A_2(n,\{d_1,d_2\})$ is upper bounded by a quadratic function:
\begin{equation}
	\label{eq:barg}
	A_2(n,\{d_1,d_2\})\le 1+{n\choose2},
\end{equation}
for all $n\ge6$. 
The result was proved by embedding a binary code into 
the sphere $S^{n-1}$ and then using bounds for spherical two-distance sets.
In fact, the bound (\ref{eq:barg}) is achieved: it is known that
$A_2(n,\{2,4\})={n\choose 2}+1$ for $n\ge6$. 
This bound is an improvement on a similar bound from \cite{LRV23}:
\[A_2(n,\{d_1,d_2\})\le{n+2\choose2},\]	
proved by a different technique. If we impose some restrictions on $d_1$ and $d_2$, it turns out that $A_2(n,\{d_1,d_2\}$ is bounded by a linear function. So, if
$d_2>2d_1$, one has $A_2(n,\{d_1,d_2\})\le n+1$. Furthermore, for $d$ even
and $\delta$ odd, one has
\[A_2(n,\{d,d+\delta\})\le\left\{
\begin{array}{ll}
	n+1 & \text{ if } d \text{ is even}, \\
	n+2 & \text{ if } d \text{ is odd}.
\end{array}
\right.\]
Let us note that the value of $A_2(n,\{d,d+\delta\})$ is of interest when both $d$ and $\delta$ are even. In case of $\delta$ odd the cardinality is bounded by a function which is linear in $n$, and the case when $d$ is odd, and $\delta$ -- even is impossible \cite{LRV23}.

I was noticed by Zinoviev \cite{Zin83} (see also \cite{LRV23,SZ69}) that the characteristic vectors of the blocks of a family of $k$-element subsets with two intersection numbers $x<y$ is a binary two distance code. In particular, quasisymmetric designs give rise to two distance codes. If there exists a 2-packing of an $n$-element set with $M$ blocks of size $k$,
then there exists a binary $(n,M,\{d,d+2\})$-code with $d=2k-2$ and cardinality
$\displaystyle M\le\frac{n(n-1)}{k(k-1)}$.
Equality is achieved iff there exists a 2-$(n,k,1)$ design. The maximal size of a 2-$(n,k,1)$
packing
is known for small values of $n$ (cf. \cite{Stinson-etal}). 
The maximal size of an $(n,M,\{d,d+2\})$-code, $d$-even, is at least equal to the maximal size of a 2-packing by $\displaystyle (\frac{d}{2}+1)$-tuples denoted by
$D(n,\frac{d}{2}+1,2)$, i.e.
\[D(n,\frac{d}{2}+1,2)\le A_2(n,\{d,d+2\})\le {n\choose2}+1.\]

The goal of this paper is to prove that for large values of $n$ and fixed $d$, 
one has $A_2(n,\{d,d+2\})=D(n,\frac{d}{2}+1,2)$. The proof relies on a classical result by Erd\H{o}s
and Hanani \cite{EH64}. For small values of $d$ ($d=4,6$) we give estimates on the smallest
$n$ for which equality takes place.

\section{Binary codes with distances $d$ and $d+2$}
\label{sec:dd+2}

In this section we consider binary codes with parameters $(n,\{d,d+2\})$ with even $d$.
As we discussed above, for $\delta$ odd $A_2(n,\{d,d+\delta\})$ is bounded by a linear function, while the case $d$ odd, $\delta$ even is impossible (we cannot have a code in which all distances are odd, and do actually appear).

We denote by $D(n,k,t)$ the maximal number of blocks in a $t$-$(v,k,1)$ packing.
A classical result by Erd\H{o}s and Hanani \cite{EH64} establishes
the asymptotic growth of $D(n,k,t)$ as $n\to\infty$. We state a special 
case of their result that will be needed in the sequel.

\begin{thm}
	\label{thm:EH64}
	Let $k\ge3$ be a fixed integer. Then
	\[\lim_{n\to\infty} D(n,k,2)\frac{k(k-1)}{n(n-1)}=1\]
\end{thm}

This result implies immediately the following corollaries.

\begin{cor}\label{cor:EH1}
	For any even integer $d$, $d\geq 4$, $$\lim_{n\rightarrow \infty}{D(n,\frac{d}{2}+1,2)\frac{(\frac{d}{2}+1)(\frac{d}{2})}{n(n-1)}}=1.$$
\end{cor}

In view of this statement it is easy to see that the following corollary takes place.

\begin{cor}{\label{cor:EH2}}
	Let $C$ be an optimal $(n, M,\{d,d + 2\})$-code. Then for any fixed $\lambda$, 
	$0< \lambda < 1$ there exists $n_0(\lambda)$ such that for any $n>n_0(\lambda)$ we have $$M\geq \frac{\lambda n^2}{(\frac{d}{2}+1)(\frac{d}{2})}.$$
\end{cor}

\begin{proof}
	Clearly $M\ge D(n,\frac{d}{2}+1,2)$ since the characteristic vectors of the blocks of an optimal 2-$(n,\frac{d}{2}+1,2)$ packing form a $(n,M,\{d,d+2\})$ code. For a  fixed constant $\lambda$ select $0<\varepsilon<1-\lambda$. There exists an integer $N_1$ such that for every $n>N_1$:
	\[-\varepsilon\frac{n(n-1)}{(\frac{d}{2}+1)\frac{d}{2}}<
	D-\frac{n(n-1)}{(\frac{d}{2}+1)\frac{d}{2}}<
	\varepsilon\frac{n(n-1)}{(\frac{d}{2}+1)\frac{d}{2}}\] 
	There exists a positive integer $N_2$ such that for $n>N_2$
	\[\frac{(1-\varepsilon-\lambda)n^2-(1-\varepsilon)n}{(\frac{d}{2}+1)\frac{d}{2}}>0.\]
	This implies that for all $n>\max\{N_1,N_2\}$:
	\[D(n,\frac{d}{2}+1,2)>(1-\varepsilon)\frac{n(n-1)}{(\frac{d}{2}+1)\frac{d}{2}}>
	\frac{\lambda n^2}{(\frac{d}{2}+1)\frac{d}{2}}.\]
\end{proof}

For our main theorem we shall need the following lemma.

\begin{lma}
	\label{L:rearrangement}	
	For a fixed even integer $d$, and $n$ sufficiently large we have that $$D(n+\frac{d}{2}-1,\frac{d}{2}+1,2) > D(n,\frac{d}{2}+1,2)+1.$$
\end{lma}
\begin{proof}
	
	Let $\mathcal{B}=\{B_i\}$ be an optimal 2-$(n,\frac{d}{2}+1,2)$ packing. Our goal is to show that if we add $\frac{d}{2}-1$ additional elements, we can construct a new packing with parameters 2-$(n+\frac{d}{2}-1,\frac{d}{2}+1,1)$ that contains two more blocks.
	
	First, we shall show that there exists a subset	$S\subseteq [n]:=\{1,\ldots,n\}$ 
	with $|S|=\frac{d}{2}+1$ such that any block of $\mathcal{B}$ meets $S$ in at most two elements. We shall construct such $S$ by an iterative procedure.
	Initially, we set $S=\varnothing$. We define also a 
	set $L$ of prohibited coordinate positions. This set $L$ is also set to be empty,
	i.e. $L=\varnothing$. Next, we repeat $\frac{d}{2}$ times the following steps:
	\begin{enumerate}
		\item Select an arbitrary element $b\in [n]\setminus L$;
		\item $S\leftarrow S\cup\{b\}$;
		\item for every element $c\in S$, $c\neq b$, check, whether $\{b,c\}$ is 
		contained in some block $B\in\mathcal{B}$. 
		\item if this is the case then $L\leftarrow L\cup B$
		\item otherwise $L\leftarrow L\cup\{b\}$.  
	\end{enumerate}
	The fact that $n$ is big enough guarantees us that at each step it is possible
	to select an element $b$.
	Finally, we end up with a set $S=\{a_1,a_2,a_3, \dots a_{\frac{d}{2}+1}\}$ such that every block from $\mathcal{B}$ contains at most two elements from $S$.  
	
	Now we are going to modify $\mathcal{B}$ in such way as to obtain
	a 2-packing $\mathcal{B}'$ of $[n+\frac{d}{2}-1]$
	with the property  that no pair of elements from $S$ is covered by the blocks of
	$\mathcal{B}'$.
	
	This is done in performing at first $\frac{d}{2}-2$ identical steps.
	For each $i=1,2, \dots \frac{d}{2}-2$ we do the following:
	
	\begin{enumerate}[]
		\item For any $j=i+1, i+2, \dots ,\frac{d}{2}+1$ consider elements $a_i, a_j\in S$; 
		if $\{a_i, a_j\}\subseteq B$ for some block $B\in\mathcal{B}$, we replace
		$B$ by $B':=(B\setminus\{a_i\})\cup\{a_n+i\}\}$.
	\end{enumerate}
	
	The obtained family of blocks still covers every pair of elements of the set $[n+\frac{d}{2}-1]$ at most once. Moreover, among the pairs from $S$ only $\{a_{\frac{d}{2}-1},a_{\frac{d}{2}} \}$,
	$\{a_{\frac{d}{2}-1},a_{\frac{d}{2}+1} \}$ and $\{a_{\frac{d}{2}},a_{\frac{d}{2}+1}\}$ may still be covered by codewords from $C'$. 
	
	If $\{a_{\frac{d}{2}-1},a_{\frac{d}{2}}\}\subseteq B_1$ for some $B_1\in\mathcal{B}$ then 
	$B_1$ is replaced by \[B_1'=(B_1\setminus\{a_{\frac{d}{2}-1}\})\cup\{a_{n+\frac{d}{2}-1}\}.\]
	If $\{a_{\frac{d}{2}-1},a_{\frac{d}{2}+1}\}\subseteq B_2$ for some $B_2\in\mathcal{B}$
	then we replace it by  
	\[B_2'= (B_2\setminus\{a_{\frac{d}{2}+1}\})\cup\{a_{n+\frac{d}{2}-1}\}.\] 
	Finally, if $\{a_{\frac{d}{2}},a_{\frac{d}{2}+1}\}\subseteq B_3$ for some $B_3$, then 
	we replace it by 
	\[C':= (B_3\setminus\{a_{\frac{d}{2}}\})\cup\{a_{n+\frac{d}{2}-1}\}.\] 
	
	The blocks $B_1, B_2, B_3$ (should there be such blocks) may intersect 
	only in the elements $a_{\frac{d}{2}-1},a_{\frac{d}{2}},a_{\frac{d}{2}+1}$. 
	The new family of blocks, denoted by $\mathcal{B}'_0$, is a 2-packing
	of $[n+\frac{d}{2}-1]$. In addition, we have achieved that
	no pair with elements from $S$ is covered by a block from $\mathcal{B}_0'$.
	Hence we can add $S$ as a block, i.e.
	$\mathcal{B}'=\mathcal{B}_0'\cup \{S\}$ is a 
	2-$(n+\frac{d}{2}-1,\frac{d}{2}+1,1)$ packing.
	
	Since $n$ is big enough we can start the whole procedure anew with 
	the only difference is that we must avoid elements that occur together 
	with elements from $\{n+1,\,n+2,\,\dots,\,n+\frac{d}{2}-1\}$ in some 
	blocks from $\mathcal{B}'$. So instead of setting $L=\varnothing$, we set
	\[L=\bigcup_{\begin{array}{c} B'\in\mathcal{B}'\\
			B'\cap\{n+1,\dots,n+\frac{d}{2}-1\}\ne\varnothing
	\end{array}} B'\]
	
	%
	%
	In this way we construct a set $S=\{a_1',a_2',a_3', \dots a_{\frac{d}{2}+1}'\}$ such that any block from
	$\mathcal{B}'$ contains at most two elements from $S$.  Moreover, blocks that
	cover some pairs from $S$ do not contain elements from $\{n+1,\,n+2,\,\dots,\,n+\frac{d}{2}-1\}$.
	Now we can repeat the argument described above to construct a new family
	$\mathcal{B}''$ that is a 2-packing of $[n+\frac{d}{2}-1]$ by blocks of size
	$\frac{d}{2}+1$ with $|\mathcal{B}''|=|\mathcal{B}|+2$.
\end{proof}

This result can be restated in terms of codes.

\begin{lma}
	Let $d$ be a fixed integer and let $n$ be another integer that is  big enough.
	Given a binary code $C$ of length $n$, minimum distance $d$ in which all 
	words are of the same weight
	$\frac{d}{2}+1$, there exists a binary constant weight code $C''$ of length
	$n+\frac{d}{2}-1$, the same minimum distance, and the same weight of the words
	for which	
	\[|C''|=|C|+2.\]   
\end{lma}

Now we state the main result in this paper.

\begin{thm}\label{Th:main}
	For any even $d>2$ there exists a positive integer $N(d)$ such that for $n\geq N(d)$ any optimal $(n, M,\{d, d + 2\})$-code $C$ is isomorphic to a constant-weight code of weight $\frac{d}{2}+1$.
\end{thm}

\begin{proof}
	For a fixed even positive integer $d$, and a fixed constant $\lambda$,
	$0<\lambda<1$, consider an optimal code  $C$ with parameters
	$(n,M,\{d_1,d_2\})$ of length $n>n_0(\lambda)$, where $n_0(\lambda)$
	is the constant from Corollary~\ref{cor:EH2}.
	Let us note that both values $d$ and $d+2$ should occur 
	as actual distances between codewords from $C$; 
	otherwise, $C$ is equidistant and $M\leq n+1$.
	
	By Corollary \ref{cor:EH2} we have that for $n>n_0(\lambda)$
	\begin{equation}{\label{eq:lowerbound}}
		M\geq \frac{\lambda n^2}{(\frac{d}{2}+1)(\frac{d}{2})}.
	\end{equation}
	
	Take two words $\vek{x}_0$ and $\vek{x}_1$ at distance $d+2$. Without loss of generality, let $\vek{x}_0=\vek{0}$ and $wt(\vek{x}_1)=d+2$.
	Since we have two possible distances, all remaining codewords belong to one of the following sets:
	\begin{eqnarray*}
		A_1 &=& \{\vek{x} \in C: wt(\vek{x})=d+2,\, d(\vek{x},\vek{x}_1)=d+2\},\\
		A_2 &=& \{\vek{x} \in C: wt(\vek{x})=d+2,\, d(\vek{x},\vek{x}_1)=d\}, \\
		A_3 &=& \{\vek{x} \in C: wt(\vek{x})=d,\, d(\vek{x},\vek{x}_1)=d+2\},\\
		A_4 &=& \{\vek{x} \in C: wt(\vek{x})=d,\, d(\vek{x},\vek{x}_1)=d\},\\
		A_5 &=& \{\vek{x}_0\},\\ 
		A_6 &=& \{\vek{x}_1\}. 
	\end{eqnarray*}
	
	Now let us consider a codeword $\vek{y}$ such that 
	$d(\vek{x}_0,\vek{y})=d(\vek{x}_1,\vek{y})=\frac{d}{2}+1$.
	We define the integers 
	\[s_i(\vek{y})=|\{\vek{z}\in\Omega^n: 
	\vek{z}\in A_i, d(\vek{y},\vek{z})=\frac{d}{2}+1\}|,\;\; i=1,2, \dots, 6.\]
	Set $B=\{\vek{y}\in\Omega^n: d(\vek{x}_0,\vek{y})=d(\vek{x}_1,\vek{y})=\frac{d}{2}+1\}$.
	Let us count in two ways the number of pairs $(\vek{y},\vek{z})$, where
	$\vek{y}\in B$, $\vek{z}\in C$. We have that 
	\begin{multline*}
		\sum_{y\in B}{\sum_{i=1}^{6}{s_i(\vek{y})}}=|A_1|+(\frac{d}{2}+2)|A_2|+
		(\frac{d}{2}+2)|A_3|+\\
		(\frac{d}{2}+1)^2|A_4|+
		{d+2 \choose \frac{d}{2}+1}|A_5|+{d+2 \choose \frac{d}{2}+1}|A_6|.
	\end{multline*}
	Since $\displaystyle M=\sum_{i=1}^6 |A_i|$ and 
	$\displaystyle |B|={d+2 \choose \frac{d}{2}+1}$, 
	there exists a word $\vek{y}_0\in B$ such that 	
	\[{\sum_{i=1}^{6}{s_i(\vek{y}_0)}}\geq \frac{M}{{d+2 \choose \frac{d}{2}+1}.}\]
	In other words, we have that the sphere with center $\vek{y}_0$ and radius
	$\displaystyle \frac{d}{2}+1$ has at least
	$\frac{M}{{d+2\choose \frac{d}{2}+1}}$ words from $C$, i.e
	\begin{equation}{\label{eq:ineq1}}
		|\{x\in C: d(y_0,x)=\frac{d}{2}+1\}|\geq \frac{M}{{d+2 \choose \frac{d}{2}+1}}.
	\end{equation}
	
	In the rest of the proof we will work with the code $C'=C+\vek{y}_0$
	which is again an optimal $(n,M,\{d,d+2\})$ code. 
	Let us define $L=\{\vek{x}\in C': wt(\vek{x})=\frac{d}{2}+1\}$. It is clear that $\vek{0}\not\in C'$.  By the choice of $\vek{y}_0$, we have 
	$|L|\geq \frac{M}{{d+2 \choose \frac{d}{2}+1}}$, i.e. 
	the size of $L$ is lowerbounded by a quadratic function on $n$.
	
	Let $\vek{z}_0\in C'$ be a codeword of maximal possible weight, denoted by $w$. 
	Clearly, $w \in \{\frac{3}{2}d+3,\frac{3}{2}d+1, \dots, \frac{d}{2}+3, \frac{d}{2}+1 \}.$
	Now we  shall consider the different possibilities for $w$.
	
	\begin{enumerate}
		\item Let $w=\frac{3}{2}d+3$.  Evidently, for any $\vek{x}\in L$ we have $\supp(\vek{x})\subset\supp(\vek{z}_0)$. 
		Every pair of coordinate positions from $\supp(\vek{z}_0)$ is covered at most once by the support of a codeword from $L$. Hence, 
		\[|L|\leq \frac{(\frac{3}{2}d+3)(\frac{3}{2}d+2)}{(\frac{d}{2}+1)\frac{d}{2}  }.\] Consequently,
		\begin{equation}{\label{eq:l1}}
			\frac{(\frac{3}{2}d+3)(\frac{3}{2}d+2)}{(\frac{d}{2}+1)\frac{d}{2}}
			\ge \frac{M}{{d+2 \choose \frac{d}{2}+1}}\ge 
			\frac{\lambda n^2}{(\frac{d}{2}+1)\frac{d}{2}{d+2\choose\frac{d}{2}+1}},
		\end{equation} 
		which does not take place for $n$ big enough.
		
		\item Let  $w=\frac{3}{2}d+3-2i$, $1\leq i\leq \frac{d}{2}-1.$
		Clearly, for any $\vek{x}\in L$, we have $wt(\vek{x}*\vek{z}_0)=\frac{d}{2}+1-i$ or $wt(\vek{x}*\vek{z}_0)=\frac{d}{2}+2-i$. As before, every pair of coordinate positions from $\supp(\vek{z}_0)$ is covered at most once by the support of a codeword from $L$. Consequently, $|L|$ is bounded above by the maximal number of codewords of length $w$ of weights $\frac{d}{2}+1-i$ and $\frac{d}{2}+2-i$ with the property that any two of them have $0$ or $1$ common ones. Therefore, 
		\[|L|\leq \frac{(\frac{3}{2}d+3-2i)(\frac{3}{2}d+2-2i)}{(\frac{d}{2}+1-i)(\frac{d}{2}-i) },\]  whence we obtain 
		\begin{equation}{\label{eq:l2}}
			\frac{(\frac{3}{2}d+3-2i)(\frac{3}{2}d+2-2i)}{(\frac{d}{2}+1-i)(\frac{d}{2}-i)}
			\ge\frac{M}{{d+2 \choose \frac{d}{2}+1}},
		\end{equation} 
		which is false for $n$ big enough.
		
		\item The case $w=\frac{d}{2}+3$. Here we have that for any $\vek{x}\in L$, we have $wt(\vek{x}*\vek{z}_0)=1$ or $wt(\vek{x}*\vek{z}_0)=2$. The number of codewords from $L$ with two $1$'s in the coordinates from
		$\supp(\vek{z}_0)$ is bounded above by $\frac{(\frac{d}{2}+3)(\frac{d}{2}+2)}{2}$
		due to the fact that two words of weight $\frac{d}{2}+1$ cannot have two common ones. 
		On the other hand, the number of words from $L$ that have exactly one $1$ inside $\supp(\vek{z}_0)$ is $\frac{n-(\frac{d}{2}d+3))(\frac{d}{2}+3)}{\frac{d}{2}}$. 
		This is due to the fact that the number of words $\vek{x}$ with 
		$wt(\vek{x}*\vek{z}_0)=1$ where this unit in the same position from
		$\supp(\vek{z}_0)$ is equal to $(n-(\frac{d}{2}+3))/\frac{d}{2}$
		Summing up we obtain   
		\[|L|\leq \frac{(\frac{d}{2}+3)(\frac{d}{2}+2)}{2 } + \frac{(n-(\frac{d}{2}d+3))(\frac{d}{2}+3)}{\frac{d}{2}}.\] 
		As a result we get 
		\begin{equation}{\label{eq:l3}}
			\frac{M}{{d+2 \choose \frac{d}{2}+1}}\leq \left(\frac{(\frac{d}{2}+3)(\frac{d}{2}+2)}{2 } + \frac{(n-(\frac{d}{2}+3))(\frac{d}{2}+3)}{\frac{d}{2}}\right),
		\end{equation} 
		which does not hold for large $n$ since $M$ is quadratic in $n$ while the 
		function on the right is linear in $n$.
	\end{enumerate}
	
	We have proved so far that there are no codewords in $C'$ of weight greater than $\frac{d}{2}+1$ for $n$ big enough. Clearly, there are also no codewords of weight less than $\frac{d}{2}-1$, because the minimal distance is $d$. 
	However, it is possible that $C'$ contains just one vertex of weight $\frac{d}{2}-1$. Again, if there are at least two of them then the distance between them is at most $d-2$. Lemma \ref{L:rearrangement} guarantees that in this case we can delete this codeword and after appropriate changing of $C'$ construct a code with cardinality $M+1$ which contradicts to the optimality of $C'$. Consequently, $C'$ is a constant-weight code.
	This argument finishes the proof.
\end{proof}

\begin{cor}
	For fixed even $d$ and $n$ big enough we have that 
	$$A_2(n, \{d, d+2\})=D(n,\frac{d}{2}+1,2).$$
\end{cor}

From this point on $N(d)$  will denote the smallest positive integer with the property
that for all $n\ge N(d)$ every optimal 
$(n,M,\{d,d+2\})$-code is a constant weight code.

A linear code is an almost-constant weight code if the weights of its
non-zero words are contained in an interval of length $\alpha$. 
It is known that for $\alpha=2$ and $k\ge4$ all almost-constant are trivial.
They are obtained from several copies of the Simplex code by adding or removing of two arbitrary non-zero coordinates (cf. \cite{LRS19}). 
It is instructive to compare this result with Theorem~\ref{Th:main}
which says that for large $n$ a (non-linear) code with distances $d$ and $d+2$ is a constant weight code.

\section{The cases $d=4$ and $d=6$}

In the previous section, we proved that there exists an integer $N(d)$ such that for 
$n\ge N(d)$ it holds 	$$A_2(n, \{d, d+2\})=D(n,\frac{d}{2}+1,2).$$ However the proof relies on an asymptotic result, so that it is
difficult to give an explicit estimate of $N(d)$. Even if we have one, inequalities (\ref{eq:l1}), (\ref{eq:l2}) and (\ref{eq:l3}) require $n$ to be at least  exponential in $d$. 
However, for small values of $d$ we can give reasonable estimates on
$N(d)$ from the values of $D(n,\frac{d}{2}+1,2)$. Below we consider the cases $d=4$ and $d=6$.

Packings by triples and quadruples yield codes with parameters
$(n,\{4,6\})$ and $(n,\{6,8\})$. 
The following bounds are well-known (see for instance, \cite{Stinson-etal})):
\[D(v,3,2)=\left\{\begin{array}{cl}
	\frac{v^2-v}{6} & \text{ if } v\equiv 1,3\pmod{6},\\
	\frac{v^2-2v}{6} & \text{ if } v\equiv 0,2\pmod{6},\\
	\frac{v^2-2v-2}{6} & \text{ if } v\equiv 4\pmod{6},\\
	\frac{v^2-v-8}{6} & \text{ if } v\equiv 5\pmod{6},
\end{array}\right.\]
Also, for all $v\notin\{8,9,10,11,17,19\}$
\[D(v,4,2)=\left\{\begin{array}{cl}
	\frac{v^2-v}{12} & \text{ if } v\equiv 1,4\pmod{12},\\
	\frac{v^2-3v}{12} & \text{ if } v\equiv 0,3\pmod{12},\\
	\frac{v^2-2v}{12} & \text{ if } v\equiv 2,8\pmod{12},\\
	\frac{v^2-2v-3}{12} & \text{ if } v\equiv 5,11\pmod{12},\\
	\frac{v^2-v-18}{12} & \text{ if } v\equiv 7,10\pmod{12},\\
	\frac{v^2-3v-6}{12} & \text{ if } v\equiv 6,9\pmod{12}.
\end{array}\right.\]


From these exact values one can easily deduce the following lemma.

\begin{lma}{\label{L:step}}
	The following inequalities take place
	\begin{enumerate}
		
		\item For any integer $v\geq 7$ we have that $D(v+1,3,2)\geq D(v,3,2)+2$.
		\item For any integer $v\geq 19$ we have that $D(v+3,4,2)\geq D(v,4,2)+2$.
	\end{enumerate}
	
\end{lma}
The proof is straightforward.

The next theorem exploits the fact that in the cases $d=4$ and $d=6$ we know the exact value of $D(n,\frac{d}{2}+1,2)$ which enables us to get reasonably small estimates on $N(4)$ and $N(6)$.

\begin{thm}{\label{Th:second_main}}
	Let $N(d)$ be the smallest integer such that for every $n\geq N(d)$ any optimal 
	$(n, M,\{d, d + 2\})$-code $C$  is (after a corresponding translation) a constant-weight code of weight $\frac{d}{2}+1$. Then $N(4)\le302$ and $N(6)=\le1685$.
\end{thm}

\begin{proof}
	The beginning of the proof coincides with the one from Theorem \ref{Th:main}. Since we know exact values of $D(n,\frac{d}{2}+1,2)$ for $d=4$ and $d=6$, instead of inequality (\ref{eq:lowerbound}) we have that
	\begin{equation}{\label{eq:l4}}
		M\geq D(n,\frac{d}{2}+1,2).
	\end{equation}
	Let us consider the case $d=4$. From inequality (\ref{eq:l4}) follows that $M\geq \frac{n^2-2n-2}{6}$. Therefore, (\ref{eq:l1}), (\ref{eq:l2}) and 
	(\ref{eq:l3}) can be rewritten in the following way: $\frac{n^2-2n-2}{6}\leq 240$, $\frac{n^2-2n-2}{6}\leq 420$ and $\frac{n^2-2n-2}{6}\leq 50(n-1)$ correspondingly. Starting from $n=302$ all these inequalities do not take place and an optimal code does not contain codewords of weight more than $3$. By the fact that $C$ is optimal and by Lemma \ref{L:step} it is guaranteed that there are no codewords of weight $1$, too. Therefore, $C$ is a constant-weight code.   
	
	The last case is $d=6$. Inequality (\ref{eq:l4}) gives us $M\geq \frac{n^2-3n-6}{12}$. Again we take inequalities (\ref{eq:l1}), (\ref{eq:l2}) and (\ref{eq:l3}) (the second one gives us two inequalities for $d=6$):  $\frac{n^2-3n-6}{12}\leq 770$, $\frac{n^2-3n-6}{12}\leq 1050$, $\frac{n^2-3n-6}{12}\leq 1960$ and $\frac{n^2-3n-6}{12}\leq 140n + 210$. Starting from $n=1685$ all these inequalities fail and an optimal code does not contain codewords of weight more than $5$. Again we use Lemma \ref{L:step} to show that $C$ is a constant-weight code.   
\end{proof}

It is expected that this estimates are rather rough and the 
actual value of $N(d)$ is much smaller.

\section*{Acknowledgements}

The first author was supported by
the Bulgarian National Science Research Fund
under Contract KP-06-N72/6 - 2023.
The research of the second author was supported 
by the NSP P. Beron project CP-MACT.


\end{document}